\documentclass[11pt,twoside]{amsart} 
\usepackage[OT2,T2A]{fontenc}
\usepackage{amssymb}
\usepackage{amsmath}
\advance \topmargin by -\headheight
\advance \topmargin by -\headsep
\evensidemargin \oddsidemargin
\date{}
\newtheorem{theorem}{Theorem}[section]
\newtheorem{corollary}[theorem]{Corollary}
\newtheorem{lemma}[theorem]{Lemma}
\newtheorem{proposition}[theorem]{Proposition}
\theoremstyle{remark} 
\newtheorem{remark}[theorem]{Remark}
\title{0-cycles on Grassmannians as representations of projective groups} 
\author{R.Bezrukavnikov} 
\address{MIT, office 2-470, 77 Massachusetts ave, Cambridge MA 02139 USA \& 
Institute for Information Transmission Problems of Russian Academy of Sciences}
\email{bezrukav@math.mit.edu}
\author{M.Rovinsky} 
\address{AG Laboratory, HSE University, Russian Federation, 6 Usacheva str., Moscow, Russia, 119048 \& 
Institute for Information Transmission Problems of Russian Academy of Sciences}
\email{marat@mccme.ru}
\begin{document} 
\dedicatory{{\fontencoding{OT2}\selectfont Rafailu Kalmanovichu Gordinu}}
\thanks{The study has been funded within the framework of the HSE University Basic Research Program 
and the Russian Academic Excellence Project '5-100'. R.B. is partially supported by an NSF grant.} 
\begin{abstract} Let $F$ be an infinite division ring, $V$ be a left $F$-vector space, $r\ge 1$ be an integer. 
We study the structure of the representation of the linear group $\mathrm{GL}_F(V)$ in the vector space of 
formal finite linear combinations of $r$-dimensional vector subspaces of $V$ with coefficients in a field $K$. 

This gives a series of natural examples of irreducible infinite-dimensional representations of projective 
groups. These representations are non-smooth if $F$ is locally compact and non-discrete. \end{abstract} 

\maketitle 

Let $F$ be a division ring (a.k.a. a skew field), and $V$ be a left $F$-vector space. 
Then $V$ is a right module over the associative unital `matrix'\ ring $\mathrm{End}_F(V)$. 

Assume that $\dim V=r+r'>1$ for a pair of cardinals $\mathbf{r}=(r,r')$. Denote by 
$\mathrm{Gr}(\mathbf{r},V)$ the set of all $F$-vector subspaces of $V$ of dimension $r$ 
and of codimension $r'$ ($\mathbf{r}$-{\sl subspaces} for brevity). 
If $r<\dim V+1$ we set $\mathrm{Gr}(r,V):=\mathrm{Gr}(\mathbf{r},V)$ for $r'=\dim V-r$. For instance, 
$\mathrm{Gr}(1,V)$ is the projective space $\mathbb P(V):=F^{\times}\backslash(V\setminus\{0\})$; 
$\mathrm{Gr}(0,V)$ and $\mathrm{Gr}((\dim V,0),V)$ are points. 

For any associative ring $A$, denote by $A[\mathrm{Gr}(\mathbf{r},V)]$ the set 
of all finite formal linear combinations $\sum_{j=1}^Na_j[L_j]$ with coefficients 
$a_j$ in $A$ of $F$-vector subspaces $L_j$ of $V$ in $\mathrm{Gr}(\mathbf{r},V)$. 

The set $A[\mathrm{Gr}(\mathbf{r},V)]$ is endowed with the evident structure of a left $A$-module. 

Let $G:=\mathrm{GL}_F(V):=\mathrm{Aut}_F(V)$ be the group of invertible elements of $\mathrm{End}_F(V)$. 
The natural right $G$-action on $\mathrm{Gr}(\mathbf{r},V)$ is transitive and gives rise to an $A$-linear 
$G$-action on $A[\mathrm{Gr}(\mathbf{r},V)]$. 

Obviously, the right module $A[\mathrm{Gr}(\mathbf{r},V)]$ admits a proper submodule 
$A[\mathrm{Gr}(\mathbf{r},V)]^{\circ}$ formed by all finite formal linear combinations 
$\sum_ja_j[L_j]$ with $\sum_ja_j=0$, which is nonzero if $r\neq 0$ and $r'\neq 0$. 

Our goal is to describe, for any coefficient field $K$, 
the structure of the right $K[G]$-module $K[\mathrm{Gr}(\mathbf{r},V)]$. 
Namely, for any infinite $F$, we show (in Theorem \ref{general_filtration}) that a canonical 
nonzero submodule $M_1$ (constructed in Lemma~\ref{2-trans-impl-all}) of 
$\mathbb Z[\mathrm{Gr}(\mathbf{r},V)]^{\circ}$ has the 
property that $K\otimes M_1$ is the only simple submodule of each nonzero $K[G]$-submodule of 
$K[\mathrm{Gr}(\mathbf{r},V)]$. 

The module $K[\mathrm{Gr}(\mathbf{r},V)]^{\circ}$ coincides with $K\otimes M_1$ if and 
only if either $r$ or $r'$ is finite. 

Several remarks on the case of finite $F$ are collected in \S\ref{finite_base_field}. 

When $F$ is a local field, one usually studies either unitary or {\sl smooth} (i.e. with open stabilizers; 
they are called {\sl algebraic} in \cite{bz}) representations, while the representations considered here 
are non-smooth. However, the latter representations are smooth if $F$ is discrete and $r$ is finite; 
for a field $F$, they arise as direct summands of ``restrictions'' of certain geometrically meaningful 
representations of automorphisms groups of universal domains over $F$, cf. \cite[\S4]{H90}. 

\section{Generators of $A[\mathrm{Gr}(r,V)]^{\circ}$ for an integer $r$} \label{Generators}
For any ring $A$ and any set $\Gamma$, denote by $A[\Gamma]$ the set of all finite formal linear 
combinations $\sum_{j=1}^Na_j[g_j]$ with coefficients $a_j$ in $A$ of elements $g_j\in\Gamma$. 

If $\Gamma$ is a group, we consider $A[\Gamma]$ as associative ring with evident relations $[g][g']=[gg']$, 
$a[g]=[g]a$ for all $g,g'\in\Gamma$ and $a\in A$. The element $[1]$ is the unit of the ring. 

The $A[G]$-module $A[\mathrm{Gr}(\mathbf{r},V)]=A\otimes\mathbb Z[\mathrm{Gr}(\mathbf{r},V)]$ 
is generated by $[L]$ for any $L\in\mathrm{Gr}(\mathbf{r},V)$. 

The following lemma shows that, for an integer $r$, the $A[\mathrm{GL}_F(V)]$-module 
$A[\mathrm{Gr}(r,V)]^{\circ}=A\otimes\mathbb Z[\mathrm{Gr}(r,V)]^{\circ}$ is generated 
by $[L]-[L']$ for any $L,L'\in\mathrm{Gr}(r,V)$ with $\dim(L\cap L')=r-1$. 
\begin{lemma} \label{2-trans-impl-all} Let $\mathbf{r}$ be a pair of cardinals. 
Let $L,L'$ be $\mathbf{r}$-subspaces in $V$ with $\dim(L/L\cap L')=\dim(L'/L\cap L')=1$. Then the 
$G$-submodule $M_1=M_1(\mathbf{r},V)$ of 
$\mathbb Z[\mathrm{Gr}(\mathbf{r},V)]^{\circ}$ generated by the difference $[L]-[L']$ contains 
all differences $[L_0]-[L_1]$ of $\mathbf{r}$-subspaces with 
$\dim(L_0/L_0\cap L_1)=\dim(L_1/L_0\cap L_1)<\infty$, but does not contain differences $[L_0]-[L_1]$ 
of other pairs of $\mathbf{r}$-subspaces. 

In particular, $M_1$ coincides with $\mathbb Z[\mathrm{Gr}(\mathbf{r},V)]^{\circ}$ 
if and only if at least one of $r$ and $r'$ is finite. \end{lemma} 
\begin{proof} Let $c=\dim(L_0/L_0\cap L_1)$. 
Fix complete flags $E_0=0\subset E_1\subset E_2\subset\cdots\subset E_c=L_0/(L_0\cap L_1)$ 
and $F_0=0\subset F_1\subset F_2\subset\cdots\subset F_c=L_1/(L_0\cap L_1)$ and set 
$L_i'=\tilde{E}_{c-i}+\tilde{F_i}$, where $\tilde{E}$ denotes the preimage of 
a subspace $E\subseteq V/(L_0\cap L_1)$ under the projection $V\to V/(L_0\cap L_1)$. 
Then $L_0',L_1',\dots,L_c'$ are $\mathbf{r}$-subspaces, while $L_{i-1}'\cap L_i'$ is 
a hyperplane in both $L_{i-1}'$ and $L_i'$ for each $i$, $1\le i\le c$. 

As $G$ acts transitively on the set of pairs $(S,S')$ of $\mathbf{r}$-subspaces of $V$ with 
$\dim(S/S\cap S')=\dim(S'/S\cap S')=1$, all $[L_i']-[L_{i+1}']$ belong to the $G$-orbit of 
$[L]-[L']$. As $[L_0]-[L_1]=([L_0']-[L_1'])+([L_1']-[L_2'])+\cdots+([L_{c-2}']-[L_{c-1}'])+([L_{c-1}']-[L_c'])$, 
we see that $M_1$ contains $[L_0]-[L_1]$. 

If either of $r$ and $r'$ is finite then for any pair $L_0,L_1$ of $\mathbf{r}$-subspaces of $V$ one has 
$\dim(L_0/L_0\cap L_1)=\dim(L_1/L_0\cap L_1)<\infty$, so it is clear from the above that $[L]-[L']$ 
generates the $G$-module $\mathbb Z[\mathrm{Gr}(\mathbf{r},V)]^{\circ}$. 

If $[L_0]-[L_1]\in M_1$, i.e., $[L_0]-[L_1]=\sum_{i=1}^Na_i([L_i']-[L_i''])$ with 
$\dim(L_i'/L_i'\cap L_i'')=\dim(L_i''/L_i'\cap L_i'')=1$, then rename $L_i'$ and $L_i''$ in 
a way to get a sequence $L_0=L_0',L_1',\dots,L_{n-1}',L_n'=L_1$ with 
$\dim(L_{i-1}'/L_{i-1}'\cap L_i')=\dim(L_i'/L_{i-1}'\cap L_i')=1$. Then $\dim(L_0'/\bigcap_{i=0}^nL_i')\le n$, 
and thus, $\dim(L_0/L_0\cap L_1)$ is finite. 

In particular, if $[L]-[L']$ generates $\mathbb Z[\mathrm{Gr}(\mathbf{r},V)]^{\circ}$ then 
$L_0/(L_0\cap L_1)$ is finite-dimensional for any pair $L_0,L_1$ of $\mathbf{r}$-subspaces 
of $V$, so at least one of $r$ and $r'$ should be finite. \end{proof}

\section{(Endo)morphisms and decomposability} \label{morphisms} 
Let $F$ be a division ring, $V$ be a left $F$-vector space, 
$\mathbf{r}_0=(r_0,r_0'),\mathbf{r}_1=(r_1,r_1')$ be two pairs of cardinals such that 
$r_0+r_0'=r_1+r_1'=\dim V$; we may omit $r_i'$ if $r_i<\dim V+1$. 
For an $\mathbf{r}_0$-subspace $L$ in $V$, denote by $\mathrm{St}_{[L]}$ the stabilizer of the point 
$L\in\mathrm{Gr}(\mathbf{r}_0,V)$ in the group $G:=\mathrm{GL}_F(V)$. 

It is easy to see that the $G$-orbit of a $F$-vector subspace $L$ in $V$ is 
determined by the pair of cardinals $(\dim L,\dim V/L)$; the $G$-orbit of 
a pair of $F$-vector subspaces $L,L'$ in $V$ is determined by the quintuple of cardinals 
$(\dim(L\cap L'),\dim L/(L\cap L'),\dim L'/(L\cap L'),\dim V/L,\dim V/L')$. 

Let $A$ be an associative unital ring. For each triple of cardinals $\mathbf{s}=(s,s',s'')$ 
with $s+s'=r_0$ and $s+s''=r_1$ (so $s'$ and $s''$ may be omitted if $s<\min(r_0,r_1)+1$), let 
\[\eta_{\mathbf{s}}^{\mathbf{r}_0,\mathbf{r}_1}:A[\mathrm{Gr}(\mathbf{r}_0,V)]\to A[\mathrm{Gr}(\mathbf{r}_1,V)]\]
be the $A[G]$-morphism given by $[L]\mapsto\sum_{L'}[L']$ if the latter sum is finite 
and non-empty, where $L'$ runs over all $\mathbf{r}_1$-subspaces in $V$ such that $\dim(L\cap L')=s$, 
$\dim L/(L\cap L')=s'$ and $\dim L'/(L\cap L')=s''$. Such $L'$'s form an $\mathrm{St}_{[L]}$-orbit. 

It is clear that $\eta_{(r,0,0)}^{\mathbf{r},\mathbf{r}}$ is identical on $A[\mathrm{Gr}(\mathbf{r},V)]$ and 
$\ker\eta_0^{\mathbf{r},0}=\ker\eta_{(r,0,r')}^{\mathbf{r},(\dim V,0)}=A[\mathrm{Gr}(\mathbf{r},V)]^{\circ}$. 

\begin{lemma} \label{Homs} Set $R:=A[G]$. 
Let $\mathbf{r}_0,\mathbf{r}_1$ be two pairs of cardinals with $r_0+r'_0=r_1+r'_1=\dim_FV$. 

Then the right $A$-module $\mathrm{Hom}_R(A[\mathrm{Gr}(\mathbf{r}_0,V)],A[\mathrm{Gr}(\mathbf{r}_1,V)])$ 
is freely generated by 

\begin{tabular}{ll} 
$\eta_{(0,r_0,0)}^{\mathbf{r}_0,\mathbf{r}_1}$ & if $\mathbf{r}_1=(0,\dim V)$;\\ 
$\eta_{(r_0,0,r_0')}^{\mathbf{r}_0,\mathbf{r}_1}$ & if $\mathbf{r}_1=(\dim V,0)$; \\ 
$\eta_{(r_0,0,r_0'-r_1')}^{\mathbf{r}_0,\mathbf{r}_1}$ & 
if $V$ is infinite, while $F$ and $r_0'\ge r_1'$ are finite; \\ 
$\eta_{(r_1,r_0-r_1,0)}^{\mathbf{r}_0,\mathbf{r}_1}$ & 
if $V$ is infinite, while $F$ and $r_0\ge r_1$ are finite; \\ 
the identity $\eta_{(r_0,0,0)}^{\mathbf{r}_0,\mathbf{r}_0}=id_{A[\mathrm{Gr}(\mathbf{r}_0,V)]}$ & 
if $V$ is infinite and $\mathbf{r}_0=\mathbf{r}_1$, \\ 
$\eta_s^{r_0,r_1}$ for all $s$, $\sigma\le s\le\min(r_0,r_1)$, & if $V$ is finite, where 
$\sigma:=\max(0,r_0-r_1')$, \\ 
$0$ & otherwise. \end{tabular} 

The ring $\mathrm{End}_{\mathbb Z[G]}(\mathbb Z[\mathrm{Gr}(\mathbf{r},V)])$ is commutative. 
If $F$ is infinite and $\mathbf{r}_0\neq\mathbf{r}_1$ then, in notation of \S\ref{Generators}, 
$\mathrm{Hom}_R(A\otimes M_1(\mathbf{r}_0,V),A\otimes M_1(\mathbf{r}_1,V))=0$, while 
$\mathrm{End}_R(A\otimes M_1(\mathbf{r}_0,V))=A$ if $r_0\neq 0$ and $r_0'\neq 0$, 
so the right $R$-modules $A[\mathrm{Gr}(\mathbf{r}_0,V)]$ and 
$A\otimes M_1(\mathbf{r}_0,V)$ are indecomposable if $A$ is a field. 
\end{lemma} 

\begin{proof} The cases $\mathbf{r}_1\in\{(0,\dim V),(\dim V,0)\}$ are trivial, since then 
$\mathrm{Gr}(\mathbf{r}_1,V)$ reduces to a single point, so we may further assume 
$\mathbf{r}_1\notin\{(0,\dim V),(\dim V,0)\}$. 

Fix some $L\in\mathrm{Gr}(\mathbf{r}_0,V)$, and suppose that the $\mathrm{St}_{[L]}$-orbit of 
a point $L'\in\mathrm{Gr}(\mathbf{r}_1,V)$ is finite. 

We, thus, assume that $L'\neq 0$ and $L'\neq V$. 
\begin{itemize} \item If $L\cong F^{\oplus r_0}$ is infinite then either (i) $L\subseteq L'$ or (ii) 
$L\cap L'=0$. In the case (ii), $L'\subseteq L$, since adding different elements of $L$ to a basis element 
of $L'$ one gets different $L'$'s, and therefore, $L'=0$. If, in the case (i), $V/L\cong F^{\oplus r_0'}$ 
is finite then all $L'$ containing $L$ form a single finite orbit; if $V/L\cong F^{\oplus r_0'}$ 
is infinite and $L'\neq V$ then $L'=L$. 
\item If $V/L\cong F^{\oplus r_0'}$ is infinite then the image of $L'$ in $V/L$ should be either 0 or 
$V/L$, i.e., either (i) $L'\subseteq L$ or (ii) $V=L+L'$. In the case (i), either (a) $L'=L$, or (b) 
$L\cong F^{\oplus r_0}$ is finite and then all $L'$ contained in $L$ form a single finite orbit. In the 
case (ii), either (a) $L'\supseteq L$, or (b) $L\cap L'\neq L$. (iia): $L'=V$, which is excluded. (iib): 
any orbit is infinite. Namely, choose a vector $v\in L\setminus L'$ and a collection $\{e_i\}_{i\in I}$ 
presenting a basis of $V/L$; then the subspaces $L_j:=L\cap L'\dotplus\langle e_i^j~|~i\in I\rangle_F$, 
where $e_i^j=e_i$ if $i\neq j$, while $e_i^i=e_i+v$, are paiwise distinct. 
\item If $V$ is finite then all orbits are finite and they are parametrized by $s=\dim(L\cap L')$, 
where $\sigma:=\max(0,r_0-r_1')\le s\le\min(r_0,r_1)$. \end{itemize}

As the $R$-module $A[\mathrm{Gr}(\mathbf{r}_0,V)]$ is generated by $[L]$, any $R$-module morphism 
$A[\mathrm{Gr}(\mathbf{r}_0,V)]\to A[\mathrm{Gr}(\mathbf{r}_1,V)]$ is determined by the image 
of $[L]$, which in turn is an element of $A[\mathrm{Gr}(\mathbf{r}_1,V)]^{\mathrm{St}_{[L]}}$, i.e., 
a linear combination of sums of the elements of several finite $\mathrm{St}_{[L]}$-orbits in 
$\mathrm{Gr}(\mathbf{r}_1,V)$. 

One has $\eta_{s'}^{r',r''}\eta_s^{r,r'}[L]=\sum_{L'}\eta_{s'}^{r',r''}[L']=
\sum_{L'}\sum_{L''}[L'']=\sum_{L''\in\mathrm{Gr}(r'',V)}N_{L,L''}[L'']$, where 
\[N_{L,L''}=|\{L'\in\mathrm{Gr}(r',V)~|~\dim(L\cap L')=s,~\dim(L'\cap L'')=s'\}|.\] 
It follows that $\eta_{s'}^{r,r}\eta_s^{r,r}=\eta_s^{r,r}\eta_{s'}^{r,r}$. In other words, the algebra 
$\mathrm{End}_R(A[\mathrm{Gr}(r,V)])$ is commutative if $V$ is finite, as soon as so is $A$. If 
$V$ is infinite then $\mathrm{End}_R(A[\mathrm{Gr}(\mathbf{r},V)])=A$. 

The $R$-module $A\otimes M_1(\mathbf{r}_0,V)$ is generated by $[L]-[L']$ for any 
$L,L'\in\mathrm{Gr}(\mathbf{r}_0,V)$ with $\dim L/(L\cap L')=\dim L'/(L\cap L')=1$, 
so any morphism $\varphi$ from the $R$-module $A\otimes M_1(\mathbf{r}_0,V)$ is 
determined by the image of $[L]-[L']$, which in turn is an element of 
$(A[\mathrm{Gr}(\mathbf{r}_1,V)]^{\circ})^{\mathrm{St}_{[L]}\cap\mathrm{St}_{[L']}}$. 

If $F$ is infinite then the only proper subspaces in $V$ fixed by $\mathrm{St}_{[L]}\cap\mathrm{St}_{[L']}$ 
are $L,L',L\cap L'$, while the $\mathrm{St}_{[L]}\cap\mathrm{St}_{[L']}$-orbits of other proper subspaces 
are infinite. This means that $\varphi([L]-[L'])=a[L]+b[L']+c[L\cap L']$ for some $a,b,c\in A$. 
Consider $g\in G$ such that $g(L)\subset L+L'$, $g(L)\notin\{L,L'\}$, $g(L\cap L')=L\cap L'$ 
and $g(L')=L'$. Then $\varphi([L]-[g(L)])=\varphi([L]-[L'])-g\varphi([L]-[L'])=a([L]-[g(L)])$. 
As $\dim L/(L\cap g(L))=\dim g(L)/(L\cap g(L))=1$, the element $[L]-[g(L)]$ is another generator of 
$A\otimes M_1$, so we get $\mathrm{End}_R(A\otimes M_1(\mathbf{r}_0,V))=A$ and the required vanishing 
for $\mathbf{r}_0\neq\mathbf{r}_1$. 
\end{proof} 

\begin{remark} \label{duality_finite_V} If $V$ is finite then the morphism $\eta_s^{r,r'}$ of 
\S\ref{morphisms} is dual to the morphism $\eta_s^{r',r}$ under the non-degenerate symmetric bilinear 
pairing on $K[\mathrm{Gr}(\bullet,V)]$, given by $([L],[L])=1$ and $([L],[L'])=0$ if $L\neq L'$: 
$((\eta_s^{r,r'})^{\ast}[L'],[L])=([L'],\eta_s^{r,r'}[L])$ is 1 if $\dim(L\cap L')=s$, 
and is 0 otherwise. \end{remark}

\section{The one-dimensional case} 
\subsection{The case of $\mathrm{SL}_2(\mathbb Q)$} 
\begin{lemma} \label{SL_2Q} Let $V$ be a two-dimensional $\mathbb Q$-vector space, $K$ be a field. 
Then the $K[\mathrm{SL}(V)]$-module $K[\mathbb{P}(V)]^{\circ}$ is simple. \end{lemma} 
\begin{proof} Fix a point $\infty\in\mathbb{P}(V)$. Let $P$ 
be the stabilizer of $\infty$ in $\mathrm{PSL}(V)$. The restriction of $K[\mathbb{P}(V)]$ to $P$ 
splits as $K[\mathbb{P}(V)\setminus\{\infty\}]\oplus K\cdot[\infty]$. Moreover, a choice of a point 
$O\in\mathbb{P}(V)\setminus\{\infty\}$ identifies the restriction of $K[\mathbb{P}(V)\setminus\{\infty\}]$ 
to the unipotent radical $P^u\cong\mathbb Q$ of $P\cong\mathbb Q^{\times}\ltimes^2\mathbb Q$ with 
the group algebra $K[P^u]$, so that the $K[P]$-submodules of $K[\mathbb{P}(V)\setminus\{\infty\}]$ 
correspond to the $P/P^u$-invariant ideals in $K[P^u]$. We identify $P^u$ with the rational powers 
$t^{\mathbb Q}$ of an indeterminate $t$, so that $K[P^u]\cong\bigcup_{N\ge 1}K[t^{1/N},t^{-1}]$. 

Any ideal in $K[P^u]$ is determined by its intersections with each 
subalgebra $K[t^{1/N},t^{-1}]$, and thus, is generated by a collection of polynomials of minimal degree 
$P_N(t^{1/N})$ for all $N\ge 1$ such that $P_N(0)=1$ and $P_N(t)|P_M(t^{N/M})$ in $K[t]$ if $M|N$. 

As $P/P^u\cong\mathbb Q^{\times}$ acts evidently on $P^u\cong t^{\mathbb Q}$, any $P/P^u$-invariant proper 
ideal $I=(P_N(t^{1/N}))_{N\ge 1}$ in $\bigcup_{N\ge 1}K[t^{1/N},t^{-1}]$ contains $P_N(t^{M_1^2/(M_2^2N)})$ 
for all $M_1,M_2,N\ge 1$. In particular, $P_N(t)$ divides $P_N(t^{M^2})$ for all $M,N\ge 1$, 
which implies that if $P_N(\alpha)=0$ then $P_N(\alpha^{M^2})=0$ for all $M\ge 1$, so the set 
$\{\alpha^{M^2}\}_{M\ge 1}$ is finite, i.e., $\alpha$ is a root of unity, say $\alpha^{M_N}=1$. 
Then $P_N(t)|(t^{M_N}-1)^{m_N}$ for some $M_N,m_N\ge 1$, and thus, $I$ contains $(t^{M_N/N}-1)^{m_N}$, 
and therefore, $I$ contains $(t^{M_N^2/N}-1)^{m_N}\in((t^{M_N/N}-1)^{m_N})$, and consequently, 
$(t^{1/N}-1)^{m_N}\in I$ for all $N\ge 1$. Let $s\ge 1$ be such that $m_s\le m_N$ for all $N\ge 1$. 
Then $(t^{1/(N^2s)}-1)^{m_s}\in I$ for all $N\ge 1$, and therefore, 
$(t^{1/N}-1)^{m_s}\in((t^{1/(N^2s)}-1)^{m_s})\subset I$ for all $N\ge 1$. 

Let $J\subset\bigcup_{N\ge 1}K[t^{1/N},t^{-1}]$ be the ideal generated by $t^{1/N}-1$ for all $N\ge 1$. 
As $J^n/J^{n+1}$ is a $K$-vector space of dimension $\max(\delta_{n,0},\delta_{\mathrm{char}(K),0})\le 1$ 
with the group $P/P^u$ acting by $n$-th powers, we conclude that $I$ coincides 
with $J^n$ for some $n\ge 0$. 

Let $M$ be a nonzero $K[\mathrm{PSL}(V)]$-submodule in $K[\mathbb{P}(V)]$. The splitting 
$K[\mathbb{P}(V)]=K\cdot[\infty]\dotplus K[\mathbb{P}(V)\setminus\{\infty\}]$ induces two projections 
$\pi_{\infty}:M\to K\cdot[\infty]$ and $\pi_{\mathrm{aff}}:M\to K[\mathbb{P}(V)\setminus\{\infty\}]$. 
As $\mathrm{PSL}(V)$ is transitive on $\mathbb{P}(V)$, the projection $\pi_{\infty}$ is surjective. 
As $K[\mathbb{P}(V)]^P=K\cdot[\infty]$, if $\pi_{\infty}$ splits as a morphism of 
$K[P]$-modules then $M$ contains $[\infty]$, and therefore, $M=K[\mathbb{P}(V)]$. 

Assume now that the projection $\pi_{\infty}$ does not split, so then (i) 
$0\to J^m\to M\to K\cdot[\infty]\to 0$, (ii) the projection $\pi_{\mathrm{aff}}$ is injective, i.e., 
the $K[P]$-module $M$ is isomorphic to a power of $J$. Since the quotient of the 
$K[P]$-module $M$ by $J^m$ is $K$, we get $m=1$, so $M\cong K[\mathbb{P}(V)\setminus\{\infty\}]$ 
as $K[P]$-module, and thus, the $K[\mathrm{PSL}(V)]$-module $M$ is generated by a nonzero 
combination of two points. This implies that $M$ contains 
$M=K[\mathbb{P}(V)]^{\circ}$. \end{proof} 

\subsection{The case of $\mathrm{SL}_2(F)$ in equal positive characteristics} 
\begin{lemma} \label{F_char_p} Let $F$ and $K$ be fields of characteristic $p>0$, and $V$ be a 
two-dimensional $F$-vector space. Then the $K[\mathrm{PSL}(V)]$-module $K[\mathbb{P}(V)]^{\circ}$ is simple. 
\end{lemma} 
\begin{proof} Let $M\neq 0$ be a $K[\mathrm{PSL}(V)]$-submodule in $K[\mathbb{P}(V)]^{\circ}$, and 
$\alpha=\sum_{i=1}^na_i[x_i]\in M$ for some $n\ge 1$, pairwise distinct $x_i\in\mathbb{P}(V)$ and 
some $a_i\in K^{\times}$ with $\sum_{i=1}^na_i=0$. Let $P\cong F^{\times}\ltimes^2F$ 
be the stabilizer of $x_1$ in $\mathrm{PSL}(V)$. The restriction of $K[\mathbb{P}(V)]$ 
to $P$ splits as $K[\mathbb{P}(V)\setminus\{x_1\}]\oplus K\cdot[x_1]$. 

A choice of a point $O\in\mathbb{P}(V)\setminus\{x_1\}$ identifies the $K[P]$-module 
$K[\mathbb{P}(V)\setminus\{x_1\}]$ with the group algebra $K[P^u]$ of the unipotent radical 
$P^u\cong F$ of $P$. Denote by $\beta=\sum_{i=2}^na_i[x_i-O]\in K[P^u]$ 
the image of $\alpha-a_1[x_1]$ under this identification. Then $\beta^{p-1}\alpha=a_1^p\cdot([x_1]-[O])$. 
As $O$ can be chosen to be an arbitrary point of $\mathbb{P}(V)\setminus\{x_1\}$, $\alpha$ generates 
$K[\mathbb{P}(V)]^{\circ}$ as $K[\mathrm{PSL}(V)]$-module. \end{proof}

\section{The case of an infinite base skew field} 
\begin{lemma} \label{1+t_rank_1} Let $V$ be a left vector space over a division ring $F$, 
and $\lambda$ be an $F$-linear functional on $V$. 
Then, for any $w\in V$, the endomorphism $1-\lambda tw\in\mathrm{End}_F(V)$, $x\mapsto x-\lambda(x)tw$, 
is not invertible for at most one value of $t\in F$. \end{lemma} 
\begin{proof} Let $\mu:=\lambda(w)$. Then, for any $t\in F$ such that $\mu t\neq 1$, one has 
$(1-\lambda tw)(1+\lambda t(1-\mu t)^{-1}w)=1-\lambda tw+\lambda t(1-\mu t)^{-1}(1-\mu t)w=1$. \end{proof} 

\begin{lemma} \label{matrix_D_is_invert} For each integer $s,N\ge 1$, let 
$D_s(N):=\sum_{m\equiv s\bmod p}(-1)^m\binom{N}{m}\in\mathbb Z$. Then the matrix \[\Delta=\begin{bmatrix} 
D_1(N)&\cdots&D_{p-1}(N)\\ \cdots&\cdots&\cdots\\ D_1(N+p-2)&\cdots&D_{p-1}(N+p-2)\end{bmatrix}\] 
is invertible over $\mathbb Z[1/p]$. \end{lemma} 
\begin{proof} Denote by $\mu_p$ the set of complex $p$-th roots of unity. Then 
\[D_s(N)=\frac{1}{p}\sum_{\zeta\in\mu_p}\zeta^{-s}(1-\zeta)^N.\] 

Fix a primitive $\zeta\in\mu_p$. Then the matrix $p\Delta$ coincides with the product 
\[\begin{bmatrix} (1-\zeta)^N&(1-\zeta^2)^N&\cdots&(1-\zeta^{p-1})^N\\ 
(1-\zeta)^{N+1}&(1-\zeta^2)^{N+1}&\cdots&(1-\zeta^{p-1})^{N+1}\\ \cdots&\cdots&\cdots&\cdots\\ 
(1-\zeta)^{N+p-2}&(1-\zeta^2)^{N+p-2}&\cdots&(1-\zeta^{p-1})^{N+p-2}\end{bmatrix} 
\begin{bmatrix} \zeta^{-1}&\zeta^{-2}&\cdots&\zeta^{-(p-1)}\\ 
\zeta^{-2}&\zeta^{-4}&\cdots&\zeta^{-2(p-1)}\\ \cdots&\cdots&\cdots&\cdots\\ 
\zeta^{-(p-1)}&\zeta^{-2(p-1)}&\cdots&\zeta^{-(p-1)(p-1)}\end{bmatrix},\] 
so $\det(p\Delta)$ is product of Vandermonde determinants: \[\det(p\Delta)=\prod_{j=1}^{p-1}(1-\zeta^j)^N 
\prod_{1\le j<s\le p-1}(\zeta^j-\zeta^s)\prod_{j=1}^{p-1}\zeta^{-j}\prod_{1\le j<s\le p-1}(\zeta^s-\zeta^j).\] 
As the norm in the extension $\mathbb Q(\zeta)/\mathbb Q$ of the element $1-\zeta^j$ is $p$, we see that 
$\det\Delta=\pm p^m$ for an integer $m\ge 0$, so $\det\Delta$ is invertible in $\mathbb Z[1/p]$. \end{proof} 

\subsection{A filtration} For a left vector space $V$ over a division ring $F$, 
let $G:=\mathrm{GL}_F(V):=\mathrm{Aut}_F(V)$. 
\begin{theorem} \label{general_filtration} Let $\mathbf{r}=(r,r')$ be a pair of cardinals 
$\ge 1$, $F$ be an infinite division ring of characteristic $p\ge 0$, $V$ be a left 
$F$-vector space of dimension $r+r'$. Set $M_0:=\mathbb Z[\mathrm{Gr}(\mathbf{r},V)]^{\circ}$ 
and $G:=\mathrm{GL}_F(V)$. Then there is a sequence of nonzero right $G$-submodules 
$M_0\supseteq M_1\supseteq M_2\supseteq M_3\supseteq\dots$ such that 
\begin{enumerate} \item \label{0=1} $M_0/M_1$ is a free abelian group, 
vanishing if and only if at least one of $r$ and $r'$ is finite; 
\item \label{p-localiz} if $p>0$ then the natural map 
$\mathbb Z[1/p]\otimes M_n\to\mathbb Z[1/p]\otimes M_1$ is surjective for any $n\ge 1$; 
\item \label{K-otimes} for any $n\ge 1$ and any field $K$, 
the natural map $K\otimes M_n\to K\otimes M_1$ is surjective; 

\item \label{any_A} if, for an associative unital ring $A$, a right $A[G]$-submodule $M$ of 
$A\otimes M_0$ contains $\alpha=\sum_{i=0}^Na_i[L_i]$, where $L_0\not\subseteq\bigcup_{i=1}^NL_i$, 
then $M$ contains $Aa_0\otimes M_n$ for some $n$ depending on $N$. \end{enumerate} \end{theorem} 
\begin{proof} Fix an arbitrary subspace $U\subset V$ of dimension $r-1$ and of codimension $r'+1$. Fix some 
$e_0,e_1\in V$ that are $F$-linearly independent in $V/U$. For each sequence $(t_i)_{i\ge 1}$ 
in $F^{\times}$ and all $n\ge 1$, define 
$\gamma_n((t_i)_{i\ge 1}):=\sum_{I\subseteq\{1,\dots,n\}}(-1)^{|I|}[U+F(e_0+(\sum_{i\in I}t_i)e_1)]\in M_0$. 

Note that (i) the $G$-orbit of $\gamma_n((t_i)_{i\ge 1})$ is independent of a particular 
choice of $U,e_0,e_1$, (ii) $\gamma_{n+1}((t_i)_{i\ge 1})=(1-\xi)\gamma_n((t_i)_{i\ge 1})$ for any 
$\xi\in G$ identical on $U+F\cdot e_1$ and such that 
$\xi e_0=e_0+t_{n+1}e_1$, (iii) all $\gamma_n((t_i)_{i\ge 1})$ are nonzero if, e.g., $t_1,t_2,\dots$ 
are either linearly independent over the prime subfield, or all equal to 1 if $p=0$. 

For each $n\ge 1$, let $M_n$ be the $G$-submodule in $M_0$ generated by the elements 
$\gamma_n((t_i)_{i\ge 1})$ for all sequences $(t_i)_{i\ge 1}$ in 
$F^{\times}$. In particular, we have inclusions $M_n\supseteq M_{n+1}$ for all $n$. 

Then (\ref{0=1}) follows from Lemma~\ref{2-trans-impl-all}: $M_0/M_1$ is the group of formal finite linear 
combinations $\sum_ia_i[L_i]'$, where $a_i\in\mathbb Z$ and $[L]'$ are classes of `commensurable' 
$\mathbf{r}$-subspaces, i.e. $L_0\sim L_1$ if $\dim(L_0/L_0\cap L_1)=\dim(L_1/L_0\cap L_1)<\infty$. 

Set $\gamma_n:=\gamma_n(1,1,1,\dots)=\sum_{s=0}^n(-1)^s\binom{n}{s}[U+F(e_0+se_1)]$. 
As any nonzero $\gamma_1((t_i)_{i\ge 1})$ belongs to the $G$-orbit of $\gamma_1$, 
the $G$-module $M_1$ is generated by $\gamma_1$. 

If $p>0$, it follows from Lemma~\ref{matrix_D_is_invert} that the $\mathbb Z[1/p]$-submodule in 
$\mathbb Z[1/p]\otimes M_0$ generated by $\gamma_n=\sum_{s=0}^{p-1}D_s(n)[U+F(e_0+se_1)],
\gamma_{n+1},\dots,\gamma_{n+p-2}$ contains $\gamma_1$, which implies (\ref{p-localiz}). 

If $\ell:=\mathrm{char}(K)\neq p$ is a prime then $\gamma_{\ell^N}\equiv[U+F\cdot e_0]-[U+F(e_0+\ell^Ne_1)]
=g_N\gamma_1\pmod{\ell M_1}$ for any integer $N\ge 1$ and some $g_N\in G$, 
so $M_n+\ell M_1$ contains $\gamma_1$ for any $n\ge 1$, which proves (\ref{K-otimes}) for 
$K=\mathbb Z/\ell$, and thus, for all fields $K$ of characteristic $\ell$. 

By Lemmas~\ref{SL_2Q} and \ref{F_char_p}, if characteristic of $F$ is 0, or if characteristics 
of $F$ and $K$ coincide, then the $K[G]$-submodule in $K\otimes M_0$ generated by $\gamma_n$ 
coincides with $K\otimes M_1$ for any $n\ge 1$. 

As (\ref{p-localiz}) implies the case of $\mathrm{char}(K)\neq p>0$, 
this completes the proof of (\ref{K-otimes}). 

For (\ref{any_A}), we are going to show that, together with $\alpha=\sum_{i=0}^Na_i[L_i]$, any right 
$A[G]$-submodule $M$ of $A\otimes M_0$ contains $a_0\otimes\gamma_n$ for some $n$ if 
$L_0$ is not contained in $L_i$, $1\le i\le N$. 

It is a folklore result that a nonzero vector space over an infinite division ring cannot be a finite 
union of proper linear subspaces, see e.g. \cite[Theorem 1.2]{Roman}. 
Fix some $v\in L_0\setminus\bigcup_{i=1}^NL_i$, some $w\in V\setminus L_0$. For each $1\le i\le N$, 
fix some $F$-linear morphism $\xi_i:V\to F\cdot w\subset V$ vanishing on $L_i$ but not on $v$.

For each subset $I\subseteq\{1,\dots,N\}$, the image of the endomorphism $\xi_I:=\sum_{i\in I}\xi_i$ 
of $V$ is $F\cdot w$, i.e., $\xi_I=\lambda_I\cdot w$ for a linear functional $\lambda_I$ on $V$, 
so by Lemma \ref{1+t_rank_1} there is at most one value of $t$ such that $1+\lambda_Itw$ is not 
invertible. As $F$ is infinite, we may therefore replace $w$ with a nonzero multiple (equivalently, 
replace $\lambda_i$'s with a common right multiple), so that $1+\xi_I=1+\lambda_I\cdot w$ 
become invertible for all subsets $I\subseteq\{1,\dots,N\}$. 

Then the element $\Xi:=\sum_{I\subseteq\{1,\dots,N\}}(-1)^{|I|}[1+\xi_I]\in\mathbb Z[G]$ 
annihilates all $[L_i]$ for $1\le i\le N$, and 
\[\Xi[L_0]=\sum_{I\subseteq\{1,\dots,N\}}(-1)^{|I|}\left[L_0\cap\ker\xi_I\dotplus 
F\cdot\left(v+(\sum_{i\in I}\lambda_i(v))w\right)\right].\] 
In particular, (i) if $r=1$ and $\lambda_i=t_i$ for all $1\le i\le N$ then 
$g\Xi\alpha=a_0\otimes\gamma_N((t_i)_{i\ge 1})$ for any $g\in G$ such that $gv=e_0$ and $gw=e_1$; 
(ii) if $\sum_{i\in I}\lambda_i\neq 0$ for any non-empty $I$ then 
$\Xi[L_0]=[L_0]+\sum_{n=1}^Mb_n[L_n']$ for some hyperplanes $L_n'\neq L_0$ in $V':=L_0\dotplus F\cdot w$ 
and some $b_n\in\mathbb Z$. 

Set $V'':=V'/\bigcap_{n=1}^N\ker\xi_n$, which is of dimension $\le N$. Denote by ${V''}^{\vee}$ 
the dual space of $V''$. Consider the canonical identification 
$\psi:\mathrm{Gr}(\dim V''-1,V'')\xrightarrow{\sim}\mathbb P({V''}^{\vee})$, sending each hyperplane $L$ 
in $V''$ to the line of all linear functionals vanishing on $L$. Then $\psi(g^{-1}[L])=g^*\psi([L])$ for 
all $g\in\mathrm{GL}_F(V'')$ (or their lifts in $\mathrm{GL}_F(V'')$), where $(g^*\lambda)(x):=\lambda(gx)$ 
for all $\lambda\in{V''}^{\vee}$ and $g\in\mathrm{GL}_F(V'')$. 

Then $\psi$ identifies the image of $\Xi[L_0]=[L_0]+\sum_{n=1}^Mb_n[L_n']$ in 
$\mathbb Z[\mathrm{Gr}(\dim V''-1,V'')]^{\circ}$ with 
$[q_0]+\sum_{n=1}^Mb_n[q_n]\in\mathbb Z[\mathbb P({V''}^{\vee})]^{\circ}$. As we have just seen in the case 
$r=1$ (with ${V''}^{\vee}$ instead of $V$), for any sequence $(t_i)_{i\ge 1}$ in $F^{\times}$, there is an 
element $\beta\in\mathbb Z[\mathrm{GL}_F(V'')]$ (similar to the element $\Xi$) such that $\beta 
([q_0]+\sum_{n=1}^Mb_n[q_n])=\sum_{I\subseteq\{1,\dots,M\}}(-1)^{|I|}[F\cdot(e_0''+(\sum_{i\in I}t_i)e_1'')]$ 
for some $F$-linearly independent $e_0'',e_1''\in V''$. Denote by $\beta^*$ the image of $\beta$ under the 
anti-involution $\mathbb Z[\mathrm{GL}_F(V'')]\to\mathbb Z[\mathrm{GL}_F(V'')]$, $[g]\mapsto[g^{-1}]$. 
Then, for any linear combination $\beta'$ of elements of $G$ identical on 
$\mathbb P(\bigcap_{n=1}^N\ker\xi_n)$ and extending $\beta^*$, one has 
$\beta'(\Xi[L_0])=\gamma_M((t_i)_{i\ge 1})$, and thus, $\beta'(\Xi\alpha)=a_0\otimes\gamma_M((t_i)_{i\ge 1})$. 
\end{proof}

\begin{corollary} \label{two_fields_dist_char} Let $\mathbf{r}=(r,r')$ be a pair of cardinals $\ge 1$, $F$ 
be an infinite division ring, $K$ be a field, $V$ be a left $F$-vector space of dimension $r+r'$. 

Then any nonzero $K$-subrepresentation of $\mathrm{GL}_F(V)$ in $K[\mathrm{Gr}(\mathbf{r},V)]$ contains 
$K\otimes M_1$. 

In particular, $K\otimes M_1$ is the only irreducible 
$K$-subrepresentation of $\mathrm{GL}_F(V)$ in $K[\mathrm{Gr}(\mathbf{r},V)]$. 

The following conditions are equivalent: {\rm (i)} at least one of $r$ and $r'$ is finite, 
{\rm (ii)} $M_1=\mathbb Z[\mathrm{Gr}(\mathbf{r},V)]^{\circ}$, 
{\rm (iii)} $K\otimes M_1=K[\mathrm{Gr}(r,V)]^{\circ}$, 
{\rm (iv)} $K[\mathrm{Gr}(\mathbf{r},V)]^{\circ}$ is irreducible, 
{\rm (v)} $K[\mathrm{Gr}(\mathbf{r},V)]/K\otimes M_1$ is irreducible, 
{\rm (vi)} $K[\mathrm{Gr}(\mathbf{r},V)]/K\otimes M_1\cong K$. \end{corollary} 

\section{Some remarks on the case of a finite base field} \label{finite_base_field} 

There is an extensive literature on representations of finite Chevalley groups, 
see e.g. \cite{Alperin,Benson}. 
For this reason we do not treat in detail the case where $F$ is a finite field.

\subsection{The case of characteristic 0 coefficient field} 
\begin{proposition} Let $F=\mathbb F_q$ be a finite field, $V$ be a finite $F$-vector space, 
$K$ be a field where $q^n\neq q$ for $2\le n\le\dim V+1$ (e.g., 
$\mathrm{char}(K)=0$), $r\ge 1$ be an integer. Then the $K[\mathrm{PGL}(V)]$-module $K[\mathrm{Gr}(r,V)]$ 
is a sum of $\min(r,\dim V-r)+1$ pairwise distinct simple submodules. \end{proposition} 
\begin{proof} As the $K[\mathrm{PGL}(V)]$-modules are semisimple, it suffices to apply 
Lemma~\ref{Homs} asserting that the algebra $\mathrm{End}_{K[\mathrm{PGL}(V)]}(K[\mathrm{Gr}(r,V)])$ 
is commutative and of dimension $\min(r,\dim V-r)+1$ as $K$-vector space. \end{proof} 

\begin{proposition} Let $F$ be a finite field, $V$ be an infinite $F$-vector space, $K$ be a field of 
characteristic 0, $r\ge 1$ be an integer. Then both $K[\mathrm{PGL}(V)]$-modules, $K[\mathrm{Gr}(r,V)]$ and 
$K[\mathrm{Gr}((\dim V,r),V)]$, admit unique composition series, both of length $r+1$. \end{proposition}

This follows from a description of the smooth $K$-representations of 
$\mathrm{GL}(V)$ for a countable $F$-vector space $V$ given in \cite[Theorem A.17]{H90}. The 
corresponding composition series are 
\[0\subset K\otimes\ker\eta_{r-1}^{r,r-1}\subset K\otimes\ker\eta_{r-2}^{r,r-2}\subset\dots\subset 
K\otimes\ker\eta_1^{r,1}\subset K\otimes\ker\eta_0^{r,0}\subset K[\mathrm{Gr}(r,V)]\quad\mbox{and}\]  
$0\subset\Phi_1\subset\Phi_2\subset\dots\subset\Phi_{r-1}\subset\Phi_r\subset K[\mathrm{Gr}((\dim V,r),V)]$, 
where $\Phi_n=K\otimes\ker\eta_{(\dim V,0,n)}^{(\dim V,r),(\dim V,r-n)}$.

\subsection{The case of positive characteristic coefficient field} 
\begin{proposition} \label{simple_dim_1} Let $K$ be a field of characteristic $\ell$. Let $F$ 
be a union of finite fields. \begin{enumerate} 
\item \label{finite_case} Suppose that $V$ is finite and either $\dim\mathbb P(V)=1$ or 
$\ell$ is not characteristic of $F$. \begin{itemize} \item If $\ell$ does not divide $|\mathbb P(V)|$ 
then $K[\mathbb P(V)]=K[\mathbb P(V)]^{\circ}\oplus K\cdot\sum_{x\in\mathbb P(V)}[x]$ 
is the sum of two simple $K[\mathrm{PGL}(V)]$-submodules. 
\item If $\ell$ divides $|\mathbb P(V)|$ then $K\cdot\sum_{x\in\mathbb P(V)}[x]$ 
is the only simple submodule of $K[\mathbb P(V)]$, 
while $K[\mathbb P(V)]^{\circ}/K\cdot\sum_{x\in\mathbb P(V)}[x]$ 
is the only simple submodule of $K[\mathbb P(V)]/K\cdot\sum_{x\in\mathbb P(V)}[x]$. \end{itemize} 
\item \label{PV_is_infinite} If $\mathbb P(V)$ is infinite and either $\dim V=2$ or characteristic 
of $F$ is not $\ell$ then $K[\mathbb P(V)]^{\circ}$ is the only simple 
$K[\mathrm{PGL}(V)]$-submodule of $K[\mathbb P(V)]$. \end{enumerate} \end{proposition}

\begin{proof} We have to show that any $\alpha=\sum_{x\in\mathbb P(V)}a_x[x]\in K[\mathbb P(V)]$ generates 
a $K[\mathrm{PGL}(V)]$-submodule containing $K[\mathbb P(V)]^{\circ}$, whenever not all $a_x$ are equal. 
As $\mathrm{PGL}(V)$ is 2-transitive on $\mathbb P(V)$, it suffices to show that the 
$K[\mathrm{PGL}(V)]$-submodule generated by $\alpha$ contains a difference of two distinct points. 

For each $x$ with $a_x\neq 0$ fix its lift $\tilde{x}\in V$. Choose a maximal subset $B$ consisting of 
$F$-linearly independent elements among $\tilde{x}$'s. We replace $F$ with the subfield of $F$ generated 
by the coefficients of the elements $\tilde{x}$ in the base $B$, and replace $\mathbb P(V)$ with the 
projectivization of the space spanned by the $\tilde{x}$'s over the new $F$. We thus assume that 
$\mathbb P(V)$ is finite. Then we proceed by induction of the dimension $n$ of $\mathbb P(V)$. 

For each hyperplane $H\subset\mathbb P(V)$, let $U_H\subset\mathrm{PGL}(V)$ be the translation group 
of the affine space $\mathbb P(V)\setminus H$. Then $(\sum_{h\in U_H}h)\alpha=
(\sum_{x\notin H}a_x)\sum_{x\notin H}[x]+q^{\dim\mathbb P(V)}\sum_{x\in H}a_x[x]$, where $q$ is order of $F$. 

For the induction step in the case $\ell\not|q$ (and $\dim\mathbb P(V)>1$), fix some hyperplane $H$ 
containg points $y,z$ with $a_y\neq a_z$ and fix some $\eta\in\mathrm{PGL}(V)$ such that $\eta(H)=H$, 
$\eta(y)=z$ and $\eta(u)=u$ for some $u\in H$. Then $(\eta-1)(\sum_{h\in U_H}h)\alpha=
q^{\dim\mathbb P(V)}\sum_{x\in H}(a_{\eta^{-1}x}-a_x)[x]=q^{\dim\mathbb P(V)}(\cdots+(a_y-a_z)[z]+0[u])$, 
so we are reduced to the case of dimension $n-1$. 

Assume now that $\dim\mathbb P(V)=1$. \begin{enumerate} \item If $\ell$ does not divide $q+1$ then there is 
$y\in\mathbb P(V)$ such that $(q+1)a_y\neq\sum_{x\in\mathbb P(V)}a_x$. 
If $\ell$ divides $q+1$ then fix an arbitrary $y\in\mathbb P(V)$ with $a_y\neq 0$ 
(so that $(q+1)a_y=0\neq\sum_{x\in\mathbb P(V)}a_x$). Fix some involution $\xi\in\mathrm{PGL}(V)$ such that 
$\xi y\neq y$. Then $(\xi-1)(\sum_{h\in U_{\{y\}}}h)\alpha=(qa_y-\sum_{x\neq y}a_x)([\xi y]-[y])=
((q+1)a_y-\sum_{x\in\mathbb P(V)}a_x)([\xi y]-[y])$. Thus, assuming that either $\ell$ does not divide 
$q+1$ or $\sum_{x\in\mathbb P(V)}a_x\neq 0$, the $K[\mathrm{PGL}(V)]$-submodule generated by $\alpha$ 
contains $[\xi y]-[y]$, and therefore, it contains $K[\mathbb P(V)]^{\circ}$. 

Assume now that $\ell$ divides $q+1$ and $\sum_{x\in\mathbb P(V)}a_x=0$. 
Then $-a_y^{-1}(\sum_{h\in U_{\{y\}}}h)\alpha=\sum_{x\in\mathbb P(V)}[x]$. 

Assuming in addition that $\alpha\notin K\cdot\sum_{x\in\mathbb P(V)}[x]$, fix some $z\in\mathbb P(V)$ with 
$a_z\neq a_y$. Let $T\subset\mathrm{PGL}(V)$ be the torus fixing $y$ and $z$. Then $(\sum_{h\in T}h)\alpha=
(\sum_{x\neq y,z}a_x)\sum_{x\neq y,z}[x]+(q-1)a_y[y]+(q-1)a_z[z]=
-(a_y+a_z)\sum_{x\neq y,z}[x]-2a_y[y]-2a_z[z]$. 

If $\ell=2$, fix some $\xi\in\mathrm{PGL}(V)$ such that $\xi y=z$ and $\xi z\neq y$. Then 
$(\sum_{h\in T}h)\alpha=-(a_y+a_z)\sum_{x\in\mathbb P(V)}[x]+(a_y+a_z)[y]+(a_y+a_z)[z]$, so 
$(\xi-1)(\sum_{h\in T}h)\alpha=(a_y+a_z)([\xi z]-[y])$, and thus, the $K[\mathrm{PGL}(V)]$-submodule 
generated by $\alpha$ contains $[\xi z]-[y]$, and therefore, it contains $K[\mathbb P(V)]^{\circ}$. 

If $\ell\neq 2$, fix some involution $\xi\in\mathrm{PGL}(V)$ such that $\xi y=z$. Then 
$(\xi-1)(\sum_{h\in T}h)\alpha=2(a_y-a_z)([y]-[z])$, 
and thus, the $K[\mathrm{PGL}(V)]$-submodule generated by $\alpha$ contains $[y]-[z]$, and therefore, 
it contains $K[\mathbb P(V)]^{\circ}$. 
\item Any nonzero element of $K[\mathbb P(V)]$ can be considered as an element 
$\alpha=\sum_{x\in\mathbb P^1(\mathbb F_q)}a_x[x]$ for a finite subfield $\mathbb F_q\subset F$. 
Extending $\mathbb F_q$ in $F$ if necessary, we may assume that $a_x=0$ for at least one $x$. Then 
it follows from (\ref{finite_case}) that any $K[\mathrm{PGL}_2(\mathbb F_q)]$-submodule containing 
$\alpha$ contains $K[\mathbb P^1(\mathbb F_q)]^{\circ}$, and thus, any $K[\mathrm{PGL}(V)]$-submodule 
containing $\alpha$ contains $K[\mathbb P(V)]^{\circ}$. \end{enumerate} \end{proof}

\begin{proposition} \label{finite_field_equal_char} Let $F=\mathbb F_q$ be a finite field, $V$ 
be an $F$-vector space and $K$ be a field extension of $F$. Then the $K[\mathrm{PGL}(V)]$-module 
$K[\mathbb P(V)]^{\circ}$ is simple if and only if $\dim V=2$. \end{proposition} 
\begin{proof} Let $\mathrm{Sym}^n_FV:=(V^{\otimes_F^n})_{\mathfrak{S}_n}$ 
be the $n$-th symmetric power of $V$, so $\dim_F\mathrm{Sym}^n_FV=\binom{\dim V+n-1}{\dim V-1}$. 
The natural morphism of $K[\mathrm{PGL}(V)]$-modules 
$K[\mathbb{P}(V)]^{\circ}=K[(V\setminus\{0\})/F^{\times}]^{\circ}\to K\otimes_F\mathrm{Sym}^{q-1}_FV$, 
$\sum_xa_x[x]\mapsto\sum_xa_x\tilde{x}^{q-1}$, is nonzero. 
One has $\dim_KK[\mathbb{P}(V)]^{\circ}=|\mathbb{P}(V)|-1=(q^{\dim V}-q)/(q-1)$. 

If $\dim V=2$ then $\dim_KK[\mathbb{P}(V)]^{\circ}=\dim_KK\otimes_F\mathrm{Sym}^{q-1}_FV=q$. 

Assuming that $(q^n-q)/(q-1)\ge\binom{n+q-2}{n-1}$ for some $n\ge 2$, let us show that 
$(q^{n+1}-q)/(q-1)>\binom{n+q-1}{n}$. Indeed, $\binom{n+q-1}{n}=\binom{n+q-2}{n-1}(n+q-1)/n<
\binom{n+q-2}{n-1}q\le q(q^n-q)/(q-1)<(q^{n+1}-q)/(q-1)$. 

This implies that $\dim K[\mathbb{P}(V)]^{\circ}>\dim K\otimes_F\mathrm{Sym}^{q-1}_FV$ if $\dim V>2$, 
and thus, the above morphism is not injective, so $K[\mathbb{P}(V)]^{\circ}$ is not simple. 

The simplicity in the case $\dim V=2$ is shown in Lemma \ref{F_char_p}. \end{proof}

\vspace{5mm}

\noindent
{\sl Acknowledgement.} {\small We are grateful to Leonid Rybnikov and Sasha Kazilo 
for bringing us together and providing an exceptional enviroment that made our work possible. 

}

\end{document}